\documentclass[12pt]{amsart}
\usepackage{amsmath,amsthm,latexsym,amscd,amsbsy,amssymb,amsfonts,fleqn,leqno,
euscript, pb-diagram}

\numberwithin{equation}{section}

\newcommand{\I}{\mathbb I}

\newtheorem{thm}{Theorem}[section]
\newtheorem{pro}[thm]{Proposition}
\newtheorem{lem}[thm]{Lemma}
\newtheorem{cor}[thm]{Corollary}

\begin{document}


\title[$\varkappa$-metrizable compacta and superextensions]
{$\varkappa$-metrizable compacta and superextensions}

\author{Vesko  Valov}
\address{Department of Computer Science and Mathematics, Nipissing University,
100 College Drive, P.O. Box 5002, North Bay, ON, P1B 8L7, Canada}
\email{veskov@nipissingu.ca}
\thanks{Research supported in part by NSERC Grant 261914-08.
The first version of this paper was completed during the author's
stay at the Faculty of Mathematics and Informatics (FMI), Sofia
University, during the period June-July 2009. The author
acknowledges FMI  for the hospitality.}

\keywords{Dugundji spaces, $\varkappa$-metrizable spaces, spaces
with closed` binary normal subbases, superextensions}

\subjclass{Primary 54C65; Secondary 54F65}


\begin{abstract}
A characterization of $\varkappa$-metrizable compacta in terms of
extension of functions and usco retractions into superextensions is
established.
\end{abstract}

\maketitle

\markboth{}{$\varkappa$-metrizable compacta}



\section{Introduction}
In this paper we assume that all topological spaces are compact
Hausdorff and all single-valued maps are continuous.

 Our main result is a characterization of {\em $\varkappa$-metrizable}
compacta (see Theorem 1.1 below) similar to the original definition
of {\em Dugundji spaces} given by Pelczynski \cite{p} (recall that
the class of $\varkappa$-metrizable spaces was introduced by
Shchepin \cite{sc1} and it contains Dugundji spaces). This
characterization is based on a result of Shapiro \cite[Theorem
3]{sh} that a compactum $X$ is $\varkappa$-metrizable if and only if
for every embedding of $X$ in a compactum $Y$ there exists a {\em
monotone extender} $u\colon C_+(X)\to C_+(Y)$, i.e. $u(f)|X=f$ for
each $f\in C_+(X)$ and $f\leq g$ implies $u(f)\leq u(g)$. Here, by
$C(X)$ and $C_+(X)$) we denote all continuous, respectively,
continuous and non-negative functions on $X$. Following \cite{sh},
we say that a map $u\colon C(X)\to C(Y)$ is: (i) {\em monotone};
(ii) {\em homogeneous}; (iii) {\em weakly additive}, if for every
$f,g\in C(X)$ and every real number $k$ we have: (i) $u(f)\leq u(g)$
provided $f\leq g$; (ii) $u(k\cdot f)=k\cdot u(f)$; (iii)
$u(f+k)=u(f)+k$. Properties (i), (ii) and (iii) imply that for every
$f\in C(X)$ and $g\in C_+(X)$ the following hold: $u(1)=1$, $u(0)=0$
and $u(g)\geq 0$.

We also introduce and investigate the covariant functor $S$: $S(X)$
is the space of all monotone, homogeneous and weakly additive maps
$\varphi\colon C(X)\to\mathbb R$ with the pointwise convergence
topology.  It is easily seen that the map $x\to\delta_x$, $x\in X$,
where each $\delta_x$ is defined by $\delta_x(f)=f(x)$, $f\in C(X)$,
is an embedding of $X$ into $S(X)$. It appears that
$\varkappa$-metrizable compacta are exactly $S$-injective compacta
in the sense of Shchepin \cite{sc3}. The functor $S$, being a
subfunctor of the functor $O$ introduced by Radul \cite{r1}, has
nice properties. For example, $S$ is open and weakly normal (see
Section 2).

Here is the characterization of $\varkappa$-metrizable compacta
mentioned above.

\begin{thm}
For a compact space $X$ the following conditions are equivalent:
\begin{itemize}
\item[(i)] $X$ is $\varkappa$-metrizable;
\item[(ii)] For any embedding of $X$ in a compactum $Y$ there exists
a monotone, homogeneous and weakly additive extender $u\colon
C(X)\to C(Y)$;
\item[(iii)] For any embedding of $X$ in a compactum $Y$ there exists
a continuous map $r:Y\to S(X)$ such that $r(x)=\delta_x$ for all
$x\in X$.
\item[(iv)] For any embedding of $X$ in a compactum $Y$ there exists
an usco map $r:Y\to\lambda X$ such that $r(x)=\eta_x$ for all $x\in
X$;
\end{itemize}
\end{thm}

Theorem 1.1 remains valid if  $S$ is replaced by the functor $O$,
see Theorem 4.2.

Recall that supercompact spaces and superextensions were introduced
by de Groot \cite{dg}. A space  is {\em supercompact} if it
possesses a binary subbase for its closed subsets. Here, a
collection $\mathcal S$ of closed subsets of $X$ is {\em binary}
provided any linked subfamily of $\mathcal S$ has a non-empty
intersection (we say that a system of subsets of $X$ is {\em linked}
provided any two elements of this system intersect). The
supercompact spaces with binary {\em normal} subbase will be of
special interest for us. A subbase $\mathcal S$ is called normal if
for every $S_0,S_1\in\mathcal S$ with $S_0\cap S_1=\varnothing$
there exists $T_0,T_1\in\mathcal S$ such that $S_0\cap
T_1=\varnothing=T_0\cap S_1$ and $T_0\cup T_1=X$. A space $X$
possessing a binary normal subbase $\mathcal S$ is called {\em
normally supercompact} \cite{vm} and will be denoted by $(X,\mathcal
S)$.

The {\em superextension} $\lambda X$ of $X$ consists of all maximal
linked systems of closed sets in $X$. The family
$$U^+=\{\eta\in\lambda X:F\subset U\hbox{~}\mbox{for
some}\hbox{~}F\in\eta\},$$ $U\subset X$ is open, is a subbase for
the topology of $\lambda X$. It is well known that $\lambda X$ is
normally supercompact. Let $\eta_x$, $x\in X$, be the maximal linked
system of all closed sets in $X$ containing $x$. The map
$x\to\eta_x$ embeds $X$ into $\lambda X$. The book of van Mill
\cite{vm} contains more information about normally supercompact
space and superextensions, see also Fedorchuk-Filippov's book
\cite{f1}.

The paper is organized as follows. Section 2 is devoted to the
functor $S$. In Section 3 we investigate the connection between
$\varkappa$-metrizable compacta and spaces with a closed binary
subbase. As corollaries we obtain Ivanov's  results from \cite{i1}
and \cite{i2} concerning superextensions of $\varkappa$-metrizable
compacta, as well as, results of Moiseev \cite{m1} about closed
hyperspaces of inclusions. The proof of Theorem 1.1 is established
in the final Section 4.

\section{Monotone, homogeneous and weakly additive functionals}

The space $S(X)$ is a subset of the product space $\mathbb R^{C(X)}$
identifying each $\varphi\in S(X)$ with $(\varphi(f))_{f\in
C(X)}\in\mathbb R^{C(X)}$. Let us also note that, according to
\cite[Lemma 1]{r1}, every monotone weakly additive functional is a
non-expanding. So, all $\mu\in S(X)$ are continuous maps on $C(X)$.

\begin{pro}
Let $X$ be a compact space. Then $S(X)$ is a compact convex subset
of $\mathbb R^{C(X)}$ containing the space $P(X)$ of all regular
probability measures on $X$.
\end{pro}

\begin{proof}
It can be established by standard arguments that $S(X)$ is a convex
compact subset of $\mathbb R^{C(X)}$. Since $P(X)$ is the space of
all monotone linear maps $\mu\colon C(X)\to\mathbb R$ with
$\mu(1)=1$, equipped with the pointwise topology,  $P(X)$ is a
subspace of $S(X)$.
\end{proof}

Obviously, for any map $f:X\to Y$ between compact spaces the formula
$(S(f)(\mu))(h)=\mu(h\circ f)$, where $\mu\in S(X)$ and $h\in C(Y)$,
defines a map $S(f)\colon S(X)\to S(Y)$. Moreover, $S(g\circ
f)=S(g)\circ S(f)$. So, $S$ is covariant functor in the category
$\mathrm{COMP}$ of compact spaces and (continuous) maps.

We say that a covariant functor $F$ in the category $\mathrm{COMP}$
is {\em weakly normal} if it satisfies the following properties:

\begin{enumerate}
\item[(a)] $F$ preserves injective maps ($F$ is monomorphic);
\item[(b)] $F$ preserves surjective maps ($F$ is epimorphic);
\item[(c)] $F$ preserves intersections ($F(\bigcap_{\alpha\in
A}X_\alpha)=\bigcap_{\alpha\in A}F(X_\alpha)$ for any family of
closed subsets $X_\alpha$ of a given compactum $X$);
\item[(d)] $F$ is continuous ($F$ preserves limits of inverse
systems);
\item[(e)] $F$ preserves weight of infinite compacta;
\item[(f)] $F$ preserves points and the empty set.
\end{enumerate}

The above properties were considered by Shchepin \cite{sc3} who
introduced the import notion of a {\em normal functor}: a weakly
normal functor preserving preimages of sets. We already mentioned
that the functor $S$ is subfunctor of the functor $O$ of order
preserving functionals introduced by Radul \cite{r1}. Since $O$ is
weakly normal \cite{r1}, we conclude that $S$ is injective,
preserves weight, points and the empty set.

\begin{thm}
The functor $S$ is weakly normal, but not normal.
\end{thm}

\begin{proof}
We follow the corresponding proof from \cite{r1} for the functor
$O$. First, let us show $S$ is surjective. Suppose $f\colon X\to Y$
is a surjective map between compacta and $\nu\in S(Y)$. Denote by
$A$ the subspace of $C(X)$ consisting of all functions $h\circ f$,
$h\in C(Y)$. $A$ has the following properties: $0_X\in A$ and $A$
contains both $\varphi+k$ and $k\varphi$ for any constant
$k\in\mathbb R$ provided $\varphi\in A$ (such $A$ will be called a
{\em weakly additive homogeneous subspace} of $C(X)$). We define a
monotone, weakly additive and homogeneous functional $\mu'\colon
A\to\mathbb R$ by $\mu'(f\circ h)=\nu(h)$. If $\mu'$ can be extended
to a functional $\mu\in S(X)$, then $S(f)(\mu)=\nu$. So, next claim
completes the proof that $S$ is epimorphic.

{\em Claim. Let $\mu'\colon A\to\mathbb R$ be a monotone, weakly
additive and homogeneous functional. Then $\mu'$ can be extended to
a monotone, weakly additive and homogeneous functional $\mu\colon
C(X)\to\mathbb R$.}

Following the proof of \cite[Lemma 2]{r1}, we consider the pairs
$(B,\mu_B)$, where $B\supset A$ and $\mu_B$ are, respectively, a
weakly additive homogeneous subspace of $C(X)$ and a monotone weakly
additive homogeneous functional on $B$. We introduced a partial
order on these pairs by $(B,\mu_B)\leq (C,\mu_C)$ iff $B\subset C$
and $\mu_C$ extends $\mu_B$. Then, according to Zorn Lemma, there
exists a maximal pair $(B_0,\mu_0)$. If $B_0\neq C(X)$, take any
$\varphi_0\in C(X)\backslash B_0$ and let $B_0^+$ (resp., $B_0^-$)
to be the set of all $\varphi\in B_0$ with $\varphi\geq\varphi_0$
(resp., $\varphi\leq\varphi_0$). Because $\mu_0$ is monotone, there
exists $p\in\mathbb R$ such that $\mu_0(B_0^-)\leq
p\leq\mu_0(B_0^+)$. Then $B_0$ is disjoint with the set
$\{k\varphi+c:k,c\in\mathbb R\}$ and
$D=B_0\cup\{k\varphi+c:k,c\in\mathbb R\}$ is a weakly additive
homogeneous subspace of $C(X)$. Define the functional
$\mu:D\to\mathbb R$ by $\mu|B_0=\mu_0$ and $\mu(k\varphi+c)=kp+c$
for all constants $k,c$. It is easily seen that $\mu$ is a monotone,
weakly additive and homogeneous functional on $D$ extending $\mu_0$.
This contradicts the maximality of $(B_0,\mu_0)$. Hence, $B_0=C(X)$.

Therefore, $S$ is epimorphic. Combining this fact with continuity of
the functor $O$ (see \cite{r1}), we obtain that $S$ is also
continuous. The arguments from the proof of \cite[Lemma 5 and
Proposition 5]{r1} imply that $S$ preserves intersections. So, $S$
is weakly normal. The example provided in \cite{r1} that $O$ does
not preserve preimages shows (without any modification) that $S$
also does not preserve intersections.
\end{proof}

Since $S$ is monomorphic, for any $\mu\in S(X)$ the set
$supp(\mu)=\bigcap\{H\subset X:H\hbox{~}\mbox{is closed
and}\hbox{~}\mu\in S(H)\}$ is well defined and is called support of
$\mu$. Since $S$ is continuous and epimorphic, next corollary
follows from \cite[Proposition 3.5]{sc3}.

\begin{cor}
The set $S_\omega(X)=\{\mu\in S(X):supp(\mu)\hbox{~}\mbox{is
finite}\}$ is dense in $S(X)$ for any compactum $X$.
\end{cor}

\begin{thm}
A surjective map $f\colon X\to Y$ is open if and only if the map
$S(f)\colon S(X)\to S(Y)$ is open.
\end{thm}

\begin{proof}
We say that a subset $A\subset S(X)$ is $S$-convex if $A$ contains
every $\mu\in S(X)$ with $\inf A\leq\mu\leq\sup A$. Here, the
functional $\inf A:C(X)\to\mathbb R$ is defined by $(\inf
A)(h)=\inf\{\nu(h):\nu\in S(X)\}$, and similarly $\sup A$. Note that
$\inf A$ and $\sup A$ are monotone and weakly additive, but not
homogeneous. So, they are not elements of $S(X)$.

Now, using the notion of $S$-continuity one can proof this theorem
following the arguments from the proof of \cite[Theorem 1]{r2}.
\end{proof}

Because $\varkappa$-metrizable compacta are exactly the compacta
which can be represented of the limit of an inverse sigma-spectrum
with open bonding projections \cite{sc3}, we obtain the following
corollary.

\begin{cor}
A compactum $X$ is $\varkappa$-metrizable if and only if $S(X)$ is
$\varkappa$-metrizable.
\end{cor}

\begin{pro}
For every compact space $X$ there exists an embedding $i:\lambda
X\to S(X)$ such that $i(\eta_x)=\delta_x$, $x\in X$. Moreover,
$i(\lambda X)\neq S(X)$ provided $X$ is disconnected.
\end{pro}

\begin{proof}
Following the proof of \cite[Theorem 3]{sh}, one can show that for
every $\eta\in\lambda X$ and $f\in C(X)$ we have
$$\max_{F\in\eta}\min_{x\in F}f(x)=\min_{F\in\eta}\max_{x\in
F}f(x)\leqno{(1)}$$ and the equality
$\varphi_{\eta}(f)=\max_{F\in\eta}\min_{x\in F}f(x)$ defines a
monotone homogeneous and weakly additive map
$\varphi_{\eta}:C(X)\to\mathbb R$. So, $\varphi_{\eta}\in S(X)$ for
every $\eta\in\lambda X$. Moreover,
$\varphi_{\eta_x}(f)=\max_{\{F:x\in F\}}\min_{y\in F}f(y)=f(x)$ for
any $x\in X$. Therefore, we obtain a map $i:\lambda X\to S(X)$,
$i(\eta)=\varphi_{\eta}$ with $i(\eta_x)=\delta_x$, $x\in X$. To
prove the first part of the proposition we need to show that $i$ is
injective and continuous.

Let $\eta\in\lambda X$ and $W$ be a neighborhood of $\varphi_{\eta}$
in $S(X)$. We can assume that $W$ consists of all $\varphi\in S(X)$
such that $|\varphi_{\eta}(f_i)-\varphi(f_i)|<\epsilon_i$ for some
functions $f_i\in C(X)$ and $\epsilon_i>0$, $i=1,..,k$. Let
$U_i=\{x\in X: f_i(x)>\varphi_{\eta}(f_i)-\epsilon_i\}$ and
$V_i=\{x\in X:f_i(x)<\varphi_{\eta}(f_i)+\epsilon_i\}$. Because of
(1), for each $i\leq k$ there exist $F_i, H_i\in\eta$ such that
$F_i\subset U_i$ and $H_i\subset V_i$. Then
$G=\bigcap_{i=1}^{i=k}(U_i^+\cap V_i^+)$ is a neighborhood of $\eta$
in $\lambda X$, and $i(\xi)\in W$ for all $\xi\in G$. So, $i$ is
continuous.

Suppose $\eta\neq\xi$ are two elements of $\lambda X$. Then there
exist $F_0\in\eta$ and $H_0\in\xi$ with $F_0\cap H_0=\varnothing$.
Let $f\in C_+(X)$ be a function such that $f\leq 1$, $f(F_0)=1$ and
$f(H_0)=0$. Consequently,
$\varphi_{\eta}(f)=\max_{F\in\eta}\min_{x\in F}f(x)=1$ and
$\varphi_{\xi}(f)=\min_{H\in\xi}\max_{x\in H}f(x)=0$. Hence, $i$ is
injective.

The second part of this proposition follows directly because $S(X)$
is always connected (as a convex set in $\mathbb R^{C(X)}$) while
$\lambda X$ is disconnected provided so is $X$.
\end{proof}

\section{$\varkappa$-metrizable compacta and superextensions}

Recall that a compactum $X$ is openly generated \cite{sc3} if $X$ is
the limit space of an inverse sigma-spectrum with open projections.
Shchepin \cite{sc1} proved that any $\varkappa$-metrizable compactum
is openly generated. On the other hand, every openly generated
compactum is $\varkappa$-metrizable (see \cite{sc3}) which follows
from Ivanov's result \cite{i1} that $\lambda X$ is a Dugundji space
provided $X$ is openly generated. We are going to use Ivanov's
theorem in the proof of Theorem 1.1 next section. That's why we
decided to establish in this section an independent proof of a more
general fact, see Proposition 3.2. Part of this proof is Lemma 3.1
below. Recall that Lemma 3.1 was established by Shirokov \cite{s}
but his proof is based on Ivanov's result mentioned above. Our proof
of Lemma 3.1 is based on the Shchepin results \cite{sc2} about
inverse systems with open projections.

We say that a subset $X$ of $Y$ is {\em regularly embedded} in $Y$
if there exists a correspondence $e$ from the topology of $X$ into
the topology of $Y$ such that $e(\varnothing)=\varnothing$,
$e(U)\cap X=U$, and $e(U)\cap e(V)=\varnothing$ provided $U\cap
V=\varnothing$.

\begin{lem}
Let $X$ be  an openly generated compactum. Then every embedding of
$X$ in a space $Y$ is regular.
\end{lem}

\begin{proof}
It suffices to show that every embedding of $X$ in a Tychonoff cube
$\I^A$ is regular, where $|A|=w(X)$. This is true if $X$ is
metrizable, see \cite[\S21, XI, Theorem 2]{k}. Suppose the statement
holds for any openly generated compactum of weight $<\tau$, and let
$X\subset\I^A$ be an openly generated compactum of weight
$|A|=\tau$. Then $X$ is the limit space of a well ordered inverse
system $S=\{X_\alpha, p^\alpha_\beta, \beta<\alpha<\omega(\tau)\}$
such that all projections $p_\alpha\colon X\to X_\alpha$ are open
surjections and all $X_\alpha$ are  openly generated compacta of
weight $<\tau$, see \cite{sc3}. Here, $\omega(\tau)$ is the first
ordinal of cardinality $\tau$. For any $B\subset A$ let $\pi_B$ be
the projection from $\I^A$ to $\I^B$, and
$p_B=\pi_B|X:X\to\pi_B(X)$. According to Shchepin's spectral theorem
\cite{sc1}, we can assume that there exists an increasing
transfinite sequence $\{A(\alpha):\alpha<\omega(\tau)\}$ of subsets
of $A$ such that $A=\bigcup\{A(\alpha):\alpha<\omega(\tau)\}$,
$X_\alpha=p_{A(\alpha)}(X)$ and $p_\alpha=p_{A(\alpha)}$. By
\cite{sc2}, $X$ has a base $\mathcal B$ consisting of open sets
$U\subset X$  with finite rank $d(U)$. Here,
$$d(U)=\{\alpha:
p_{\alpha+1}(U)\neq\big(p^{\alpha+1}_\alpha\big)^{-1}(p_\alpha(U))\}.$$
For every $U\in\mathcal B$ let $A(U)=\{\alpha_0, \alpha,
\alpha+1:\alpha\in d(U)\}$, where $\alpha_0\in A$ is fixed.
Obviously, $X$ is a subset of
$\prod\{X_\alpha:\alpha<\omega(\tau)\}$. For every $U\in\mathcal B$
we consider the open set
$\gamma_1(U)\subset\prod\{X_\alpha:\alpha<\omega(\tau)\}$ defined by
$$\gamma_1(U)=\prod\{p_\alpha(U):\alpha\in
A(U)\}\times\prod\{X_\alpha:\alpha\in A\backslash A(U)\}.$$ Now, let
$$\gamma (W)=\bigcup\{\gamma_1(U):U\in\mathcal
B\hbox{~}\mbox{and}\hbox{~}\overline{U}\subset W\}, W\in\mathcal
T_X.$$

{\em Claim. The following conditions hold: $(i)$ $\gamma
(W_1)\cap\gamma (W_2)=\varnothing$ whenever $W_1$ and $W_2$ are
disjoint open sets in $X$; $(ii)$ $\gamma (W)\cap X=W$,
$W\in\mathcal T_X$.}

\smallskip
Suppose $W_1\cap W_2=\varnothing$. To prove condition (i), it
suffices to show that $\gamma_1(U_1)\cap\gamma_1(U_2)=\varnothing$
for any pair $U_1, U_2\in\mathcal B$ with $\overline{U_i}\subset
W_i$, $i=1,2$. Indeed, we fix such a pair and let
$\beta=\max\{A(U_1)\cap A(U_2)\}$. Then
 $\beta$ is either $\alpha_0$ or $\max\{d(U_1)\cap d(U_2)\}+1$. In both
 cases $d(U_1)\cap d(U_2)\cap [\beta,\omega(\tau))=\varnothing$.
 According to \cite[Lemma 3, section 5]{sc2}, $p_\beta(U_1)\cap
 p_\beta(U_2)=\varnothing$.
 Since $\beta\in A(U_1)\cap A(U_2)$,
 $\gamma_1(U_1)\cap\gamma_1(U_2)=\varnothing$.

Obviously, $W\subset\gamma (W)\cap X$ for all $W\in\mathcal T_X$.
So, condition $(ii)$ will follow as soon as we prove that $\gamma
(W)\cap X\subset W$. Fix $x\in\gamma(W)\cap X$ and let $U\in\mathcal
B$ such that $x\in\gamma_1(U)$ and $\overline{U}\subset W$. Define
$\beta(U)=\max d(U)+1$. Then
 $p_\alpha(x)\in p_\alpha(U)$ for all $\alpha\leq\beta(U)$.
 Since $\alpha\not\in d(U)$ for all $\alpha\geq\beta(U)$, we
have $\big(p^\alpha_{\beta(U)}\big)^{-1}(p_{\beta(U)}(x))\subset
p_\alpha(U)$ for $\alpha>\beta(U)$. Hence, $p_\alpha(x)\in
p_\alpha(U)$ for all $\alpha$. The last relation implies that
$x\in\overline{U}$, so $x\in W$. This completes the proof of the
claim.

We are going to show that $\prod\{X_\alpha:\alpha\in A\}$ is
regularly embedded in $\prod\{\I^{A(\alpha)}:\alpha\in A\}$.
According to our assumption, each $X_\alpha$ is regularly embedded
in $\I^{A(\alpha)}$, so there exists a regular operator
$e_\alpha:\mathcal T_{X_\alpha}\to\mathcal T_{\I^{A(\alpha)}}$. Let
$\mathcal B_1$ be a base for $\prod\{X_\alpha:\alpha\in A\}$
consisting of sets of the form $V=\prod\{U_\alpha:\alpha\in
\Gamma(V)\}\times\prod\{X_\alpha:\alpha\in A\backslash\Gamma(V)\}$
with $\Gamma(V)\subset A$ being a finite set. For every
$V\in\mathcal B_1$ we assign the open set
$\theta_1(V)\subset\prod\{\I^{A(\alpha)}:\alpha\in A\}$,
$$\theta_1(V)=\prod\{e_\alpha\big(U_\alpha\big):\alpha\in
\Gamma(V)\}\times\prod\{\I^{A(\alpha)}:\alpha\in A\backslash
\Gamma(V)\}.$$  Now, we define a regular operator $\theta$ between
the topologies of $\prod\{X_\alpha:\alpha\in A\}$ and
$\prod\{\I^{A(\alpha)}:\alpha\in A\}$ by
$$\theta(G)=\bigcup\{\theta_1(V):V\in\mathcal
B_1\hbox{~}\mbox{and}\hbox{~}V\subset G\}.$$ Since, according to
Claim 1, $X$ is regularly embedded in $\prod\{X_\alpha:\alpha\in
A\}$ and $\gamma$ is the corresponding regular operator, the
equality $e_1(W)=\theta(\gamma(W))$, $W\in\mathcal T_X$, provides a
regular operator between the topologies of $X$ and
$\prod\{\I^{A(\alpha)}:\alpha\in A\}$. Finally, because
$X\subset\I^A\subset\prod\{\I^{A(\alpha)}:\alpha\in A\}$, we define
$e\colon\mathcal T_X\to\mathcal T_{\I^A}$,
$$e(W)=\bigcup\{e_1(W)\cap\I^A:W\in\mathcal T_X\}.$$
So, $X$ is regularly embedded in $\I^A$.
\end{proof}

\begin{pro}
Let $X$ be an openly generated compact space possessing a binary
closed subbase $\mathcal S$. Then $X$ is a Dugundji space. Moreover,
if in addition $X$ is connected and $\mathcal S$ is normal, then $X$
is an absolute retract.
\end{pro}

\begin{proof}
Suppose $X$ is an openly generated compactum with a binary closed
subbase $\mathcal S$ and $X$ is embedded in $\I^\tau$ for some
$\tau$. By Lemma 3.1, $X$ is regularly embedded in $\I^\tau$, so
there exists a regular operator $e\colon\mathcal T_X\to\mathcal
T_{\I^\tau}$. Define the set-valued map $r\colon\I^\tau\to X$ by
$$r(y)=\bigcap\{I_{\mathcal S}(\overline{U}):y\in e(U), U\in\mathcal
T_{X}\}\hbox{~}\mbox{if}\hbox{~} y\in\bigcup\{e(U): U\in\mathcal
T_{X}\}\leqno{(2)}$$
$$\hbox{~}\mbox{and}\hbox{~} r(y)=X
\hbox{~}\mbox{otherwise},\hbox{~}$$ where $\overline U$ is the
closure of $U$ in $X$ and $I_{\mathcal S}(\overline
U)=\bigcap\{S\in\mathcal S:\overline U\subset S\}$. Since $e$ is
regular, for every $y\in\I^\tau$ the system $\gamma_y=\{U\in\mathcal
T_{X} :y\in e(U)\}$ is linked. Consequently,
$\omega_y=\{S\in\mathcal S:\overline{U}\subset S\hbox{~}\mbox{for
some}\hbox{~}U\in\gamma_y\}$ is also linked, so
$r(y)=\bigcap\{S:S\in\omega_y\}\neq\varnothing$.

It is easily seen that $r(x)=x$ for all $x\in X$. Moreover, $r$ is
usc. Indeed, let $r(y)\subset W$ with $y\in\I^\tau$ and
$W\in\mathcal T_X$. Then there exist finitely many $U_i\in\mathcal
T_{X}$, $i=1,2,..,k$, such that $y\in \bigcap_{i=1}^{i=k}e(U_i)$ and
$\bigcap_{i=1}^{i=k}I_{\mathcal S}(\overline{U}_i)\subset W$.
Obviously, $r(y')\subset W$ for all
$y'\in\bigcap_{i=1}^{i=k}e(U_i)$. So, $r$ is an usco retraction from
$\I^\tau$ onto $X$. According to \cite{dr}, $X$ is a Dugundji space.

Suppose now $X$ is connected and $\mathcal S$ is a binary normal
subbase  for $X$. By \cite{vm}, any set of the form $I_{\mathcal
S}(F)$ is $\mathcal S$-convex (a set $A\subset X$ is $\mathcal
S$-convex if $I_{\mathcal S}(x,y)\subset A$ for all $x,y\in A$). So,
each $r(y)$ is an $\mathcal S$-convex set. According to
\cite[Corollary 1.5.8]{vm}, all closed $\mathcal S$-convex subsets
of $X$ are also connected. Hence, the map $r$, defined by $(2)$, is
connected-valued. Consequently, by \cite{dr}, $X$ is an absolute
extensor in dimension 1, and there exists a map $r_1:\I^\tau \to
\exp X$ with $r_1(x)=\{x\}$ for all $x\in X$, see \cite[Theorem
3.2]{f}. Here, $\exp X$ is the space of all closed subsets of $X$
with the Vietoris topology. On the other hand, since $X$ is normally
supercompact, there exists a retraction $r_2$ from $\exp X$ onto
$X$, see \cite[Corollary 1.5.20]{vm}. Then the composition $r_2\circ
r_1\colon\I^\tau\to X$ is a (single-valued) retraction. So, $X\in
AR$.
\end{proof}

\begin{cor}$($\cite{i1},\cite{i2}$)$
$\lambda X$ is a Dugundji space provided $X$ is an openly generated
compactum. If, in addition, $X$ is connected, then $\lambda X$ is an
absolute retract.
\end{cor}

\begin{proof}
It is easily seen that $\lambda$ is a continuous functor preserving
open maps, see \cite{f1}. So, $\lambda X$ is openly generated. Since
$\lambda X$ has a binary normal subbase, Proposition 3.2 completes
the proof.
\end{proof}

A closed subset $\mathcal A$ of $\exp X$ is called an {\em inclusion
hyperspace} if for every $A\in\mathcal A$ and every $B\in\exp X$ the
inclusion $A\subset B$ implies $B\in\mathcal A$. The space of all
inclusion hyperspaces with the inherited topology from $\exp^2X$ is
closed in $\exp^2X$ and it is denoted by $GX$ \cite{m1}.

\begin{cor}\cite{m1}
$GX$ is a Dugundji space provided $X$ is an openly generated
compactum. If, in addition, $X$ is connected, then $GX$ is an
absolute retract.
\end{cor}

\begin{proof}
Since $G$ is a continuous functor preserving open maps \cite{tz},
$GX$ is openly generated if $X$ is so. On the other hand, $GX$ has a
binary subbase \cite{r0}. Hence, by Proposition 3.2, $GX$ is a
Dugundji space. Suppose $X$ is connected and openly generated. Then
$GX$ is a connected Dugundji space. Hence, $\lambda (GX)$ is an
absolute retract. Moreover, there exist natural inclusions
$GX\subset\lambda (GX)\subset G^2X=G(GX)$, and a retraction from
$G^2X$ onto $GX$, see \cite[Lemma 2]{r0}. Consequently, $GX\in AR$
as a retract of $\lambda (GX)$.
\end{proof}
\section{Proof of Theorem 1.1}

We begin this section with the following lemma.

\begin{lem}
Let $X$ be a compact regularly embedded subset of a space $Y$. Then
there exists an usco map $r\colon Y\to\lambda X$ such that
$r(x)=\eta_x$ for all $x\in X$.
\end{lem}

\begin{proof}
Let $e\colon\mathcal T_{X}\to\mathcal T_{Y}$ be a regular operator
between the topologies of $X$ and $Y$. Define $r\colon Y\to\lambda
X$ by
$$r(y)=\bigcap\{(\overline{U})^+:y\in e(U), U\in\mathcal
T_{X}\}\hbox{~}\mbox{if}\hbox{~} y\in\bigcup\{e(U): U\in\mathcal
T_{X}\},$$
$$\hbox{~}\mbox{and}\hbox{~} r(y)=\lambda X
\hbox{~}\mbox{otherwise}\hbox{~},$$ where $\overline U$ is the
closure of $U$ in $X$ and $(\overline{U})^+\subset\lambda X$
consists of all $\eta\in\lambda X$ with $\overline U\in\eta$. Since
$e$ is regular, for every $y\in Y$ the system $\{U\in\mathcal T_{X}
:y\in e(U)\}$ is linked. Consequently, $\{\overline{U}:y\in e(U)\}$
is also linked and consists of closed sets in $X$. Hence,
$\bigcap\{(\overline{U})^+:y\in e(U), U\in\mathcal
T_{X}\}\neq\varnothing$. So, $r$ has non-empty compact values. It is
easily seen that $r(x)=\eta_x$ for all $x\in X$. One can show (as in
the proof of Proposition 3.2) that $r$ is usc.
\end{proof}

{\em Proof of Theorem $1.1$} $(i)\Rightarrow (ii)$ We follow the
proof of \cite[Theorem 3]{sh}. The equality (1) provides a monotone
homogeneous and weakly additive extender $u_1\colon C(X)\to
C(\lambda X)$. Consider the space $Z$ obtained by attaching $\lambda
X$ to $Y$ at the points of $X$. Then $\lambda X$ is a closed subset
of $Z$ and, since $\lambda X$ is a Dugundji space (see Corollary
3.3), there exists a linear monotone extension operator $u_2\colon
C(\lambda X\to C(Z)$ with $u_2(1)=1$. Then $u(f)=u_2(u_1(f))|Y$,
$f\in C(X)$, defines a monotone homogeneous and weakly additive
extender $u\colon C(X)\to C(Y)$.

$(ii)\Rightarrow (iii)$ Suppose $X$ is a subset of $Y$ and $u\colon
C(X)\to C(Y)$ is a monotone homogeneous and weakly additive
extender. Define $r\colon Y\to S(X)$ by $r(y)(f)=u(f)(y)$, where
$y\in Y$ and $f\in C(X)$. It is easily seen that $r$ is continuous
and $r(x)=\delta_x$ for all $x\in X$.

$(iii)\Rightarrow (iv)$ Let $r_2\colon Y\to S(X)$ be a continuous
map with $r_2(x)=\delta_x$ for all $x\in X$. The equality
$u(f)(\varphi)=\varphi(f)$ defines a monotone extender  $u\colon
C(X)\to C(S(X))$. Then, by \cite[Theorem 4.1]{vd}, $X$ is regularly
embedded in $S(X)$. So, according to Lemma 4.1, there exists an usco
map $r_1\colon S(X)\to\lambda X$ such that $r_1(\delta_x)=\eta_x$,
$x\in X$. Then $r=r_1\circ r_2\colon Y\to\lambda X$ is the required
map.

$(iv)\Rightarrow (i)$ Let $X$ be embedded in $\I^A$ for some $A$ and
$r\colon\I^A\to\lambda X$ be an usco map with $r(x)=\eta_x$, $x\in
X$. For every open $U\subset X$ define $e(U)=\{y\in\I^A:r(y)\subset
U^+\}$. Since $U^+$ is open in $\lambda X$ and $r$ is usc, $e(U)$ is
open in $\I^A$. Moreover, $e(U)\cap X=U$ and $e(U)\cap
e(V)=\varnothing$ provided $U$ and $V$ are disjoint. So, $X$ is
regularly embedded in $\I^A$. Finally, by \cite{s}, $X$ is
$\varkappa$-metrizable.\hfill$\square$

Now, we consider the functor $O$ of order preserving functionals
introduced by Radul \cite{r1}. Recall that a map $\varphi\colon
C(X)\to\mathbb R$ is called order preserving if it is monotone,
weakly additive and $\varphi(1_X)=1$. Theorem 4.2 below follows from
Theorem 1.1. Indeed, since $S$ is a subfunctor of $O$ we have
embeddings $X\subset S(X)\subset O(X)$ for every compactum $X$. This
implies implications $(i)\Rightarrow (ii)$ and $(ii)\Rightarrow
(iii)$ of Theorem 4.2. The proofs of implications $(iii)\Rightarrow
(vi)$ and $(vi)\Rightarrow (i)$ are the same.

\begin{thm}
For a compact space $X$ the following conditions are equivalent:
\begin{itemize}
\item[(i)] $X$ is $\varkappa$-metrizable;
\item[(ii)] For any embedding of $X$ in a compactum $Y$ there exists
a monotone and weakly additive extender $u\colon C(X)\to C(Y)$ such
that $u(1_X)=1_Y$;
\item[(iii)] For any embedding of $X$ in a compactum $Y$ there exists
a continuous map $r:Y\to O(X)$ such that $r(x)=\delta_x$ for all
$x\in X$.
\item[(iv)] For any embedding of $X$ in a compactum $Y$ there exists
an usco map $r:Y\to\lambda X$ such that $r(x)=\eta_x$ for all $x\in
X$;
\end{itemize}
\end{thm}

Let $F$ be a covariant functor in the category of compact spaces and
continuous maps. A compactum $X$ is said to be $F$-injective
\cite{sc3} if for any map $f\colon Y\to X$ and any embedding
$i\colon Y\to Z$ there exists a map $g\colon F(Z)\to F(X)$ such that
$F(f)=g\circ F(i)$. It is easily seen that if $F$ has the property
that for every $X$ there exists an embedding $j_X\colon X\to F(X)$,
then $X$ is $F$-injective iff for every embedding of $X$ in another
compactum $Y$ there exists a map $h\colon Y\to F(X)$ with
$h(x)=j_X(x)$ for all $x\in X$. It follows from Theorem 1.1(iii)
that $\varkappa$-metrizable compacta are exactly the $S$-injective
compacta.

\textbf{Acknowledgments.} The author would like to express his
gratitude to the referee for several valuable remarks and
suggestions which completely changed the initial version of the
paper. The author is also indebted to Taras Radul for pointing out
that $S$ is a subfunctor of the functor $O$, and for providing his
papers treating the properties of the functor $O$.


\end{document}